\documentclass[12pt, reqno, final]{amsart}

\usepackage{amsmath,amsthm}
\usepackage{amsfonts}  
\usepackage{amssymb}   
\usepackage{latexsym}
\usepackage{amscd}
\usepackage[mathscr]{eucal}
\usepackage[all]{xy}

\usepackage[pdftex]{graphicx}
\usepackage[bookmarksnumbered, colorlinks, plainpages]{hyperref}
\hypersetup{colorlinks=true,linkcolor=red, anchorcolor=green, citecolor=cyan, urlcolor=red, filecolor=magenta, pdftoolbar=true}

\DeclareMathAlphabet\mathoo{U}{eur}{b}{n}

 \DeclareMathOperator{\Real}{Re}

\newtheorem{theorem}{Theorem}
\newtheorem{proposition}[theorem]{Proposition}
\newtheorem{corollary}[theorem]{Corollary}
\theoremstyle{definition}
\newtheorem{example}{Example}
\theoremstyle{remark}
\newtheorem{remark}{Remark}

\begin{document}
\large

\title[An approximation of matrix exponential]{An approximation of matrix exponential\\
by a truncated Laguerre series}

\author{E.~D. Khoroshikh}
 \address{Department of System Analysis and Control,
Voronezh State University\\ 1, Universi\-tet\-skaya Square, Voronezh 394018, Russia}
\email{\textcolor[rgb]{0.00,0.00,0.84}{xoroshix2002@mail.ru}}

\author{V.~G. Kurbatov$^*$}
 \address{Department of System Analysis and Control,
Voronezh State University\\ 1, Universi\-tet\-skaya Square, Voronezh 394018, Russia}
\email{\textcolor[rgb]{0.00,0.00,0.84}{kv51@inbox.ru}}
\thanks{$^*$ Corresponding author}



\date{\today}

\begin{abstract}
The Laguerre functions $l_{n,\tau}^\alpha$, $n=0,1,\dots$, are constructed from generalized Laguerre polynomials. The functions $l_{n,\tau}^\alpha$ depend on two parameters: scale $\tau>0$ and order of generalization $\alpha>-1$, and form an orthogonal basis in $L_2[0,\infty)$. Let the spectrum of a square matrix $A$ lie in the open left half-plane. Then the matrix exponential $H_A(t)=e^{At}$, $t>0$, belongs to $L_2[0,\infty)$. Hence the matrix exponential $H_A$ can be expanded in a series $H_A=\sum_{n=0}^\infty S_{n,\tau,\alpha,A}\,l_{n,\tau}^\alpha$. An estimate of the norm $\Bigl\lVert H_A-\sum_{n=0}^N S_{n,\tau,\alpha,A}\,l_{n,\tau}^\alpha\Bigr\rVert_{L_2[0,\infty)}$ is proposed. Finding the minimum of this estimate over $\tau$ and $\alpha$ is discussed. Numerical examples show that the optimal $\alpha$ is often almost 0, which essentially simplifies the problem.
\end{abstract}
\maketitle

\section*{Introduction}\label{s:intro}
An approximate calculation of the matrix exponential $H_A(t)=e^{At}$, $t>0$,
is of constant importance~\cite{Godunov94:ODE:eng,Golub-Van_Loan13:eng,Higham08,
Horn-Johnson01:eng,Moler-Van_Loan78,Moler-Van_Loan03} at least for solving linear differential equations $\dot x(t)=Ax(t)+f(t)$. In this paper, the problem of approximate representation of $H_A(t)$ in the from of a formula depending on the parameter $t$ is discussed.
The matrix $A$ is assumed to be stable, i.~e. the eigenvalues of $A$ lie in the open left half-plane.

Typically, the complete representation of $H_A(t)=e^{At}$ in the form of a formula is very cumbersome (every element of the matrix function $H_A$ is a linear combination of $M$ functions of the form $t\mapsto t^je^{\lambda_k t}$, where $M\times M$ is the size of the matrix $A$) and therefore inconvenient for large $M$. We discuss the approximate representation of $H_A$ in the form of a truncated Laguerre series; it contains fewer terms and is therefore simpler than the complete $H_A$.

There are a lot of papers devoted to the approximation of impulse responses and matrix exponential by the truncated Laguerre series, see, for example, \cite{Belt-Brinker97,Clowes65,Moore11,Prokhorov-Kulikovskikh15,
Sarakr-Koh02,Sheehan-Saad-Sidje10,Srivastava-Mavromatis-Alassar03,Tuma-Jura19} and references therein. The main difference between these papers and the present one lies in the formulation of the problem.

The \emph{generalized Laguerre polynomials} are the functions
\begin{equation*}
L_n^\alpha(t)=\frac{t^{-\alpha}e^t}{n!}\bigl(t^{n+\alpha}e^{-t}\bigr)^{(n)},\qquad\alpha>-1,\; t\ge0,\;n=0,1,\dots.
\end{equation*}
We call \emph{Laguerre functions} the modified Laguerre polynomials:
\begin{equation*}
l_{n,\tau}^\alpha(t)=\sqrt{\frac{n!}{\Gamma(n+\alpha+1)}}
\tau^{\frac{\alpha+1}{2}}\,t^{\frac{\alpha}{2}}\,e^{-\tau t/2}L_n^\alpha(\tau t),\qquad t\ge0,\;\;n=0,1,\dots,
\end{equation*}
The Laguerre functions depend on two parameters: the order of generalization $\alpha>-1$ and the scale $\tau>0$.
The most important and simple case is when $\alpha=0$.

It is known that $l_{n,\tau}^\alpha$ form an orthonormal basis in $L_2[0,\infty)$. Therefore the matrix exponential (impulse response)
\begin{equation*}
H_A(t)=e^{At},\qquad t>0,
\end{equation*}
can be expanded in the \emph{Laguerre series}
\begin{equation*}
H_{A}=\sum_{n=0}^\infty S_{n,\tau,\alpha,A}\,l_{n,\tau}^\alpha
\end{equation*}
with matrix coefficients $S_{n,\tau,\alpha,A}$. The coefficient $S_{n,\tau,\alpha,A}$ can be interpreted (Proposition~\ref{p:func cal}) as the result of the substitution of $A$ in the function $\lambda\mapsto s_{n,\tau,\alpha,\lambda}$, where
\begin{equation*}
s_{n,\tau,\alpha,\lambda}=\int_0^\infty e^{\lambda t}\,l_{n,\tau}^\alpha(t)\,dt.
\end{equation*}

Since the matrix exponential $H_{A}$ is a linear combination of functions of the form $t\mapsto t^je^{\lambda_k t}$ (where $\lambda_k$ are eigenvalues of $A$), which resemble $l_n$, it is natural to expect that the Laguerre series converges rather fast and hence its truncation or partial sum \begin{equation*}
H_{N,\tau,\alpha,A}=\sum_{n=0}^N S_{n,\tau,\alpha,A}\,l_{n,\tau}^\alpha
\end{equation*}
approximates $H_{A}$ quite well.

The aim of this paper is estimating the accuracy $\lVert H_{A}-H_{N,\tau,\alpha,A}\rVert_{L_2}$ and choosing $\tau$ and $\alpha$ that provide the best estimate. The idea of the paper is as follows.
We describe (Theorem~\ref{t:1}) the estimate of $\lVert H_{A}-H_{N,\tau,\alpha,A}\rVert_{L_2}$ in terms of the quantities
\begin{align*}
\varphi(N,\tau,\alpha)=&\sum_{k=0}^{M}\zeta(N,\tau,\alpha,\lambda_k),\\
\psi(N,\tau,\alpha)=&\max_{k}\zeta(N,\tau,\alpha,\lambda_k),
\end{align*}
where $\lambda_k$, $k=1,2\dots,M$, are eigenvalues of $A$, and
\begin{equation*}
\zeta(N,\tau,\alpha,\lambda)
=\int_0^\infty \Bigl\lvert e^{\lambda t}-\sum_{n=0}^{N}\,s_{n,\tau,\alpha,\lambda}\,l_{n,\tau}^\alpha(t)\Bigr\rvert^2\,dt.
\end{equation*}
We recommend to choose $\tau$ and $\alpha$ so that $\varphi(N,\tau,\alpha)$ be minimal (alternatively, $\psi(N,\tau,\alpha)$ could be minimised). Numerical experiments show that the optimal $\alpha$ is often close to 0. Optimization with respect to $\tau$ only significantly simplifies and speeds up calculations.

The paper is organized as follows.
Sections~\ref{s:LF} and~\ref{s:DS and IR} are devoted to definitions and notation. In Section~\ref{s:coeff}, we derive some formulas for Laguerre coefficients. In Section~\ref{s:estimate}, we describe the main estimate. In Section~\ref{s:partial tau}, we recall~\cite{Belt-Brinker97,Brinker-Belt95,Clowes65} formulas for calculating the derivative of $\zeta$ with respect to $\tau$. They simplifies the search for the minimum with respect to $\tau$ only. The problem of finding minimum over $\tau$ only arises, for example, when we restrict ourselves to the case $\alpha=0$; some simplified formulas for the case $\alpha=0$ are collected in Section~\ref{s:alpha=0}.
In Section~\ref{s:opt choice}, we present the recommending algorithm of finding $\tau$ and $\alpha$. In Section~\ref{s:num ex}, we describe the results of two numerical experiments.

We use `Wolfram Mathematica'~\cite{Wolfram} for our computer calculations.

\section{The definition of Laguerre functions}\label{s:LF}
In this section, we recall some definitions.

The functions
\begin{align*}
L_n^\alpha(t)&=\frac{t^{-\alpha}e^t}{n!}\bigl(t^{n+\alpha}e^{-t}\bigr)^{(n)}\\
&=\sum_{k=0}^n(-1)^k\binom{n+\alpha}{n-k}\frac{\lambda^k}{k!},
\qquad t\ge0,\;\;\alpha>-1,\,n=0,1,\dots.
\end{align*}
are called~\cite[p. 775]{Abramowitz-Stegun92:eng},~\cite[p. 71]{FMF63:eng},~\cite[p. 31]{Gautschi04} \emph{generalized Laguerre polynomials}.
The special cases
\begin{equation*}
L_n(t)=\frac{e^t}{n!}\bigl(t^{n}e^{-t}\bigr)^{(n)},\qquad t\ge0,\;n=0,1,\dots.
\end{equation*}
of these functions are called (\emph{ordinary}) \emph{Laguerre polynomials}.
Actually, $L_n^\alpha$ is a polynomial of degree $n$.
It is well-known~\cite[Theorem 5.7.1]{Szego75:eng},~\cite[p.~88]{Lebedev72:eng},~\cite[\S~8]{Nikiforov-Uvarov85:eng} that the functions $L_n^\alpha$ form an orthogonal basis in $L_2[0,\infty)$ with the weight function $t\mapsto t^{\alpha}e^{-t}$:
\begin{equation*}
\int_0^\infty t^\alpha e^{-t}L_n^\alpha(t)\,L_m^\alpha(t)\,dt=\frac{\Gamma(n+\alpha+1)}{n!}\,\delta_{nm},\qquad n,m=0,1,\dots,
\end{equation*}
where $\delta_{nm}$ is the Kronecker symbol.
We note that the ordinary polynomials $L_n$ are normalized, but the generalized ones $L_n^\alpha$, $\alpha\neq0$, are not.

Let $\tau>0$ be a given number. It plays a role of a time scale.
We call the family of functions

\begin{equation}\label{e:Laguerre functions:alpha}
\begin{split}
l_{n,\tau}^\alpha(t)
&=\sqrt{\frac{n!}{\Gamma(n+\alpha+1)}}
\tau^{\frac{\alpha+1}{2}}\,t^{\frac{\alpha}{2}}\,e^{-\tau t/2}L_n^\alpha(\tau t)\\
&=\sqrt{\frac{\tau n!}{\Gamma(n+\alpha+1)}}e^{-\tau t/2}\sum_{k=0}^n(-1)^k\binom{n+\alpha}{n-k}
\frac{(\tau\lambda)^{k+\alpha/2}}{k!},
\quad n=0,1,\dots.
\end{split}
\end{equation}
the (\emph{generalized}) \emph{Laguerre functions}. In particular,
\begin{equation*}%
l_{n,\tau}(t)=l_{n,\tau}^0(t)=\sqrt{\tau}\,e^{-\tau t/2}L_n(\tau t),\qquad t\ge0,\;\;n=0,1,\dots.
\end{equation*}
Evidently, the Laguerre functions $l_{n,\tau}^\alpha$ form an orthonormal basis in $L_2[0,\infty)$ (without weight):
\begin{equation*}
\int_0^\infty l_{n,\tau}^\alpha(t)\,l_{m,\tau}^\alpha(t)\,dt=\delta_{nm},\qquad t\ge0,\;\;n,m=0,1,\dots.
\end{equation*}

\section{Laguerre series for $H_A$}\label{s:DS and IR}
In this section, we introduce some notation.

Let $M$ be a positive integer.
We denote by $\mathbb C^{M\times M}$ the linear space of all matrices of the size $M\times M$; the symbol $\mathbf1\in\mathbb C^{M\times M}$ denotes the identical matrix.

Let $A\in\mathbb C^{M\times M}$ be a given matrix. We assume that $A$ is \emph{stable}, i.~e. the eigenvalues of $A$ lie in the open left half-plane. We discuss the expansion of the function
\begin{equation*}
H_A(t)=e^{At},\qquad t>0,
\end{equation*}
in the series of Laguerre functions. We call $H_A$ the \emph{matrix exponential} or the \emph{impulse response} of the differential equation
\begin{equation*}
\dot x(t)=Ax(t)+f(t).
\end{equation*}

We recall that the generalized Laguerre functions~\eqref{e:Laguerre functions:alpha}
form an orthonormal basis in $L_2[0,\infty)$; here the scale parameter $\tau>0$ and the order of generalization $\alpha>-1$ can be chosen arbitrarily. Therefore the matrix exponential $H_{A}$ can be represented in the form of the \emph{Laguerre series}
\begin{equation*}
H_{A}=\sum_{n=0}^\infty S_{n,\tau,\alpha,A}\,l_{n,\tau}^\alpha,
\end{equation*}
where the Laguerre coefficients
\begin{equation}\label{e:def of S}
S_{n,\tau,\alpha,A}=\int_0^\infty H_{A}(t)\,l_{n,\tau}^\alpha(t)\,dt
\end{equation}
are matrices.
Since the matrix exponential $H_{A}$ is a linear combination of functions of the form $t\mapsto t^je^{\lambda_k t}$, it is natural to expect that the series converges rather fast and hence its $N$-truncation
\begin{equation}\label{e:H_{N,tau,alpha,A}}
H_{N,\tau,\alpha,A}(t)=\sum_{n=0}^N S_{n,\tau,\alpha,A}\,l_{n,\tau}^\alpha(t)
\end{equation}
with relatively small $N$ approximates $H_{A}$ well enough.
Nevertheless, note that the number $N$ can not be taken very large, because of rounding errors.

The aim of this paper is to estimate the quantity
\begin{equation*}
\lVert H_{A}-H_{N,\tau,\alpha,A}\rVert_{L_2[0,\infty)}
=\sqrt{\int_{0}^{\infty}\bigl\lVert H_{A}(t)-H_{N,\tau,\alpha,A}(t)\bigr\rVert_F^2\,dt},
\end{equation*}
where $\lVert\cdot\rVert_F$ is the Frobenius norm,
and to give recommendations on the optimal choice of $\tau$ and $\alpha$ based on it.

\section{The Laguerre coefficients of $h_\lambda$}\label{s:coeff}
In the simplest case, when the matrix $A$ has the size $1\times1$, the problem of construction of approximation~\eqref{e:H_{N,tau,alpha,A}} is reduced to the calculation of Laguerre coefficients $s_{n,\tau,\alpha,\lambda}$ of the function $t\mapsto e^{\lambda t}$; we do it in Proposition~\ref{p:s_n}. Then we describe the expression of $S_{n,\tau,\alpha,\lambda}$ in terms of $s_{n,\tau,\alpha,\lambda}$ (Proposition~\ref{p:func cal}).

For $\Real\lambda<0$ (here and below $\Real$ means the real part of a complex number), we consider the auxiliary function
\begin{equation*}
h_\lambda(t)=e^{\lambda t},\qquad t>0.
\end{equation*}
It is straightforward to verify that
\begin{equation*}
\lVert h_\lambda\rVert_{L_2[0,\infty)}=\frac1{\sqrt{-2\Real\lambda}}.
\end{equation*}

Our interest in the function $h_\lambda$ is explained by the following.
If $\lambda$ is an eigenvalue of $A$ (recall that $\Real\lambda<0$) and $v$ is the corresponding normalized eigenvector, then the function
\begin{equation*}
x_\lambda(t)=H_{A}(t)v
\end{equation*}
can be represented as
\begin{equation*}
x_\lambda(t)=h_{\lambda}(t)v.
\end{equation*}

Let us first perform some calculations with the functions $h_{\lambda}$. They can be interpreted as the approximation of the matrix exponential $H_{A}$ by the truncated Laguerre series~\eqref{e:H_{N,tau,alpha,A}} when $A$ is a matrix of the size $1\times1$ whose only element equals $\lambda$.

We denote by $s_{n,\tau,\alpha,\lambda}$ the Laguerre coefficients of the function $h_{\lambda}$ in the orthonormal basis $l_{n,\tau}^\alpha$:
\begin{equation}\label{e:s via int}
s_{n,\tau,\alpha,\lambda}=\int_0^\infty h_{\alpha,\lambda}(t)\,l_{n,\tau}^\alpha(t)\,dt,\qquad \Real\lambda<0.
\end{equation}
Clearly, $s_{n,\tau,\alpha,\lambda}$ are real for real $\lambda$. Therefore, from the Schwarz reflection principle for holomorphic functions~\cite[theorem 7.5.2]{Greene-Krantz06}, it follows that
\begin{equation}\label{e:bar s}
\overline{s_{n,\tau,\alpha,\lambda}}=s_{n,\tau,\alpha,\bar\lambda},
\end{equation}
where the bar means the complex conjugate. Representation~\eqref{e:bar s} is useful for symbolic calculation of derivatives.

\begin{proposition}\label{p:s_n}
Let $\Real\lambda<0$. Then
\begin{equation}\label{e:s_{n,tau,alpha,lambda}}
\begin{split}
s_{n,\tau,\alpha,\lambda}
&=\frac{\Gamma(\alpha/2+1)}{(\tau/2-\lambda)^{\alpha/2+1}}\,
\tau^{\frac{\alpha+1}{2}}\,
\binom{n+\alpha}{n}\,\sqrt{\frac{n!}{\Gamma(n+\alpha+1)}}\,\\
&\times{}_2F_1\bigl(-n,\alpha/2+1,\alpha+1,\tau/(\tau/2-\lambda)\bigr),
\end{split}
\end{equation}
where ${}_2F_1$ is the hypergeometric function.
In particular,
\begin{equation*}
s_{n,\tau,0,\lambda}=-\frac{2\sqrt{\tau} (2\lambda+\tau)^n}{(2\lambda-\tau)^{n+1}},\qquad n=0,1,\dots.
\end{equation*}
\end{proposition}

\begin{remark}\label{r:hypergeometric}
We note that the function $z\mapsto{}_2F_1\bigl(-n,\alpha/2+1,\alpha+1,z\bigr)$ is a polynomial of degree $n$, since~\cite[p.~10]{FMF63:eng} its first argument $-n$ is a negative integer. Thus it is calculated quickly and accurately.
\end{remark}

\begin{proof}
We begin with the formula~\cite[formula (16)]{Srivastava-Mavromatis-Alassar03}
\begin{align*}
\int_0^\infty t^\beta\,e^{-\sigma t}\,L_n^\alpha(\tau t)\,L_k^\beta(\sigma t)\,dt
&=\binom{n+\alpha}{n-k}\binom{k+\beta}{k}\,\frac{\tau^k\,\Gamma(\beta+1)}{\sigma^{\beta+k+1}}\\
&\times{}_2F_1\bigl(-n+k,\beta+k+1,\alpha+k+1,\tau/\sigma\bigr).
\end{align*}
We have (see~\eqref{e:Laguerre functions:alpha} and note that $L_0^\beta(t)=1$ for all $t\ge0$ and $\binom{\alpha/2}{0}=1$)
\begin{align*}
s_{n,\tau,\alpha,\lambda}
&=\int_0^\infty e^{\lambda t}\,\sqrt{\frac{n!}{\Gamma(n+\alpha+1)}}\,
\tau^{\frac{\alpha+1}{2}}\,t^{\frac{\alpha}{2}}\,e^{-\tau t/2}\,L_n^\alpha(\tau t)\,dt\\
&=\sqrt{\frac{n!}{\Gamma(n+\alpha+1)}}\,
\tau^{\frac{\alpha+1}{2}}\,
\int_0^\infty e^{(\lambda-\tau/2)t}\,t^{\frac{\alpha}{2}}\,L_n^\alpha(\tau t)\,dt\\
&=\sqrt{\frac{n!}{\Gamma(n+\alpha+1)}}\,
\tau^{\frac{\alpha+1}{2}}\,
\int_0^\infty t^{\frac{\alpha}{2}}\,e^{(\lambda-\tau/2)t}\,L_n^\alpha(\tau t)\,L_0^{\alpha/2}\bigl((\tau/2-\lambda) t\bigr)\,dt\\
&=\sqrt{\frac{n!}{\Gamma(n+\alpha+1)}}\,
\tau^{\frac{\alpha+1}{2}}\,
\binom{n+\alpha}{n}\,\frac{\Gamma(\alpha/2+1)}{(\tau/2-\lambda)^{\alpha/2+1}}\\
&\times{}_2F_1\bigl(-n,\alpha/2+1,\alpha+1,\tau/(\tau/2-\lambda)\bigr).\qed
\end{align*}
\renewcommand\qed{}
\end{proof}

\begin{remark}\label{r:Mavromatis}

In a similar way one can derive the formula for the Laguerre coefficients of the functions $t\mapsto t^je^{\lambda t}$ which correspond to generalized eigen\-vec\-tors of $A$:
\begin{align*}
q_{n,\tau,\alpha,\lambda}
&=\int_0^\infty t^je^{\lambda t}\,l_{n,\tau}^\alpha(t)\,dt
=\frac{\Gamma(\alpha/2+j+1)}{(\tau/2-\lambda)^{\alpha/2+1}}\,
\tau^{\frac{\alpha+1}{2}}\,
\binom{n+\alpha}{n}\,\\
&\times\sqrt{\frac{n!}{\Gamma(n+\alpha+1)}}\,
{}_2F_1\bigl(-n,\alpha/2+j+1,\alpha+1,\tau/(\tau/2-\lambda)\bigr).
\end{align*}
\end{remark}

\begin{proposition}\label{p:func cal}
Let the spectrum of $A$ lie in the open left half-plane. Then
the coefficient $S_{n,\tau,\alpha,A}$ is the function $\lambda\mapsto s_{n,\tau,\alpha,\lambda}$ of $A$.
\end{proposition}
\begin{proof}
Let $n\in\mathbb N_0$, $\tau>0$, and $\alpha>-1$ be fixed.
For brevity, we set $f(\lambda)=s_{n,\tau,\alpha,\lambda}$. From Proposition~\ref{p:s_n} it is seen that $f$ is holomorphic in the open left half-plane $\Real\lambda<0$.
We recall that
\begin{equation*}
s_{n,\tau,\alpha,\lambda}=\int_0^\infty h_{\lambda}(t)\,l_{n,\tau}^\alpha(t)\,dt
=\int_0^\infty e^{\lambda t}\,l_{n,\tau}^\alpha(t)\,dt.
\end{equation*}
From this formula, it is clear that
\begin{equation*}
\frac{\partial s_{n,\tau,\alpha,\lambda}}{\partial \lambda}=t\,s_{n,\tau,\alpha,\lambda},\qquad
\frac{\partial^j s_{n,\tau,\alpha,\lambda}}{\partial \lambda^j}=t^j\,s_{n,\tau,\alpha,\lambda}.
\end{equation*}

We recall~\cite[ch.~VII, \S~1, Theorem 5]{Dunford-Schwartz-I:eng},~\cite[ch.~1, \S~5]{Kato95:eng} that for any square matrix $A$,
\begin{equation}\label{e:mat func}
f(A)=\sum_{k=1}^m\sum_{j=0}^{w_k-1}\frac{\partial^j f}{\partial \lambda^j}(\lambda_k)\frac{N_k^{j}}{j!},
\end{equation}
where $\lambda_k$ are eigenvalues of $A$, $w_k$ are their multiplicities, and $N_k$ are spectral nilpotents, in particular, $N_k^{0}=P_k$ are spectral projectors. If all eigenvalues are simple, then
\begin{equation*}
f(A)=\sum_{k=1}^M f(\lambda_k)P_k.
\end{equation*}
For the exponential function $\lambda\mapsto e^{\lambda t}$ formula~\eqref{e:mat func} takes the form
\begin{equation}\label{e:mat exp}
e^{At}=\sum_{k=1}^m\sum_{j=0}^{w_k-1}
t^je^{\lambda_kt}\frac{N_k^{j}}{j!}.
\end{equation}
In particular, if all eigenvalues are simple, then
\begin{equation*}
e^{At}=\sum_{k=1}^M e^{\lambda_kt}P_k.
\end{equation*}

From~\eqref{e:mat exp} and~\eqref{e:def of S} it follows that
\begin{align*}
S_{n,\tau,\alpha,A}&=\int_0^\infty e^{At}\,l_{n,\tau}^\alpha(t)\,dt\\
&=\int_0^\infty \sum_{k=1}^m\sum_{j=0}^{w_k-1}
t^je^{\lambda_kt}\frac{N_k^{j}}{j!}\,l_{n,\tau}^\alpha(t)\,dt\\
&=\sum_{k=1}^m\sum_{j=0}^{w_k-1}\frac{N_k^{j}}{j!}
\int_0^\infty t^je^{\lambda_kt}\,l_{n,\tau}^\alpha(t)\,dt\\
&=\sum_{k=1}^m\sum_{j=0}^{w_k-1}\frac{N_k^{j}}{j!}
\,\frac{\partial^j s_{n,\tau,\alpha,\lambda_k}}{\partial \lambda^j},
\end{align*}
which, by~\eqref{e:mat func}, equals the function $\lambda\mapsto s_{n,\tau,\alpha,\lambda}$ of $A$.
\end{proof}

We do not discuss algorithms for calculating of the coefficients $S_{n,\tau,\alpha,A}$ in this paper. One possibility for this purpose is
Proposition~\ref{p:func cal} and Corollary~\ref{c:reccur r_n:func cal} (see below). When $S_{n,\tau,\alpha,A}$ are found, we obtain the approximation
\begin{equation*}
H_A(t)\approx \sum_{n=0}^N S_{n,\tau,\alpha,A}\,l_{n,\tau}^\alpha(t).
\end{equation*}

\section{The estimate of accuracy}\label{s:estimate}
In order for the truncated Laguerre series~\eqref{e:def of S} approximate the matrix exponential $H_A$ well enough, first of all, the truncated Laguerre series
\begin{equation*}
h_{N,\tau,\alpha,\lambda}=\sum_{n=0}^N s_{n,\tau,\alpha,\lambda}\,l_{n,\tau}^\alpha
\end{equation*}
should approximate the function $h_\lambda$ for all $\lambda\in\sigma(A)$. In this section, we discuss the inverse problem: how to estimate $\bigl\lVert H_{A}-H_{N,\tau,\alpha,A}\bigr\rVert_{L_2}$ in terms of $\bigl\lVert h_{\lambda_k}-h_{N,\tau,\alpha,\lambda_k}\bigr\rVert_{L_2}$, where $\lambda_k$ runs over the eigenvalues of $A$.

For $\Real\lambda<0$ and a natural number $N$ we denote by $\zeta(N,\tau,\alpha,\lambda)$ the square of the accuracy of the approximation of the function $h_{\lambda}$ by its $N$-truncated Laguerre series:
\begin{equation}\label{e:func zeta}
\zeta(N,\tau,\alpha,\lambda)
=\int_0^\infty \Bigl\lvert e^{\lambda t}-\sum_{n=0}^{N}\,s_{n,\tau,\alpha,\lambda}\,l_{n,\tau}^\alpha(t)\Bigr\rvert^2\,dt.
\end{equation}
Clearly, we can rewrite this formula as
\begin{equation}\label{e:zeta as a sum from N+1}
\begin{split}
\zeta(N,\tau,\alpha,\lambda)
&=\Bigl\lVert h_{\lambda}-\sum_{n=0}^{N}s_{n,\tau,\alpha,\lambda}\,l_{n,\tau}^\alpha\Bigr\rVert_{L_2}^2\\
&=\Bigl\lVert\sum_{n=N+1}^{\infty}s_{n,\tau,\lambda}\,l_{n,\tau}\Bigr\rVert^2\\
&=\sum_{n=N+1}^{\infty}|s_{n,\tau,\alpha,\lambda}|^2.
\end{split}
\end{equation}

For a matrix $C=\{C_{ij}\}\in\mathbb C^{M\times M}$, we denote by $\lVert C\rVert_{2\to2}$ the norm induced by the Euclidean norm $\lVert\cdot\rVert_2$ on $\mathbb C^M$ and by
\begin{equation*}
\lVert C\rVert_F=\sqrt{\sum_{i=1}^M\sum_{j=1}^M|C_{ij}|^2}
\end{equation*}
the Frobenius norm~\cite[p.~71]{Golub-Van_Loan13:eng}.
It is easy to show that
\begin{align*}
\lVert A\rVert_{2\to2}&\le\lVert A\rVert_F,\\
\lVert AB\rVert_F&\le\lVert A\rVert_{2\to2}\cdot\lVert B\rVert_F,\\
\lVert AB\rVert_F&\le\lVert A\rVert_F\cdot\lVert B\rVert_{2\to2},\\
\lVert Ax\rVert_2&\le\lVert A\rVert_F\cdot\lVert x\rVert_{2}.
\end{align*}
By default, we use for matrices $C\in\mathbb C^{M\times M}$ the Frobenius norm. We denote by $\sigma(C)$ the spectrum (the set of all eigenvalues) of a square matrix $C$. We recall that eigenvalues and eigenvectors are calculated~\cite{Wolfram} with high backward stability by the QR-algorithm~\cite{Golub-Van_Loan13:eng,Higham08,Voevodin-Kuznetsov:rus-eng}, so we suppose that they are known.

\begin{proposition}\label{p:G_D-G_D,N}
Let $D\in\mathbb C^{M\times M}$ be a diagonal matrix with diagonal elements $\lambda_k$, $\Real\lambda_k<0$, $k=1,2,\dots,M$. Then for the number
\begin{equation*}
\lVert H_D-H_{N,\tau,\alpha,D}\rVert_{L_2}=\sqrt{\int_{0}^{\infty}\bigl\lVert H_D(t)-H_{N,\tau,\alpha,D}(t)\bigr\rVert_F^2\,dt},
\end{equation*}
we have
\begin{equation*}
\lVert H_D-H_{N,\tau,\alpha,D}\rVert_{L_2}=\sqrt{\sum_{k=0}^{M}\zeta(N,\tau,\alpha,\lambda_k)}
\le\sqrt{M\max_{k}\zeta(N,\tau,\alpha,\lambda_k)},
\end{equation*}
where $\lambda_k$ are the eigenvalues of $D$ and the function $\zeta$ is defined by~\eqref{e:func zeta}.
\end{proposition}
\begin{proof}
By assumption, the matrix $D$ has the form
\begin{equation*}
D=\begin{pmatrix}
\lambda_1& 0&\ldots& 0\\
0& \lambda_2&\ldots& 0\\
\vdots& \vdots& \ddots& \vdots\\
0& 0&\ldots& \lambda_M
\end{pmatrix}.
\end{equation*}
Therefore,
\begin{equation*}
H_D(t)=\begin{pmatrix}
h_{\lambda_1}(t)& 0&\ldots& 0\\
0& h_{\lambda_2}(t)&\ldots& 0\\
\vdots& \vdots& \ddots& \vdots\\
0& 0&\ldots& h_{\lambda_M}(t)
\end{pmatrix}
\end{equation*}
and
\begin{multline*}
H_{N,\tau,\alpha,D}(t)\\
=\begin{pmatrix}
\sum_{n=0}^{N}s_{n,\tau,\alpha,\lambda_1}\,l_{n,\tau}^\alpha(t)& 0&\ldots& 0\\
0& \sum_{n=0}^{N}s_{n,\tau,\alpha,\lambda_2}\,l_{n,\tau}^\alpha(t)&\ldots& 0\\
\vdots& \vdots& \ddots& \vdots\\
0& 0&\ldots& \sum_{n=0}^{N}s_{n,\tau,\alpha,\lambda_M}\,l_{n,\tau}^\alpha(t)
\end{pmatrix}.
\end{multline*}
Hence, by the definition of the Frobenius norm,
\begin{equation*}
\bigl\lVert H_D(t)-H_{N,\tau,\alpha,D}(t)\bigr\rVert_F^2=
\sum_{k=0}^{M}\Bigl\lvert h_{\lambda_k}(t)-\sum_{n=0}^{N}s_{n,\tau,\alpha,\lambda_k}\,l_{n,\tau}^\alpha(t)\Bigr\rvert^2.
\end{equation*}
Consequently, (recall that $\Real\lambda_k<0$)
\begin{multline*}
\sqrt{\int_{0}^{\infty}\bigl\lVert H_D(t)-H_{N,\tau,\alpha,D}(t)\bigr\rVert_F^2\,dt}\\
=\sqrt{\int_{0}^{\infty}\sum_{k=0}^{M}\Bigl\lvert h_{\lambda_k}(t)-\sum_{n=0}^{N}s_{n,\tau,\alpha,\lambda_k}\,l_{n,\tau}^\alpha(t)\Bigr\rvert^2\,dt}\\
=\sqrt{\sum_{k=0}^{M}\int_{0}^{\infty}\Bigl\lvert h_{\lambda_k}(t)-\sum_{n=0}^{N}s_{n,\tau,\alpha,\lambda_k}\,l_{n,\tau}^\alpha(t)\Bigr\rvert^2\,dt}\\
=\sqrt{\sum_{k=0}^{M}\zeta(N,\tau,\alpha,\lambda_k)}.\qed
\end{multline*}
\renewcommand\qed{}
\end{proof}

Now let us suppose that the matrix $A$ is \emph{diagonalizable}; this means that there exists an invertible matrix $T$ and a diagonal matrix $D$ such that
\begin{equation*}
A=TDT^{-1}.
\end{equation*}
In such a case, the diagonal elements of $D$ are the eigenvalues of $A$ and the columns of $T$ are the corresponding eigenvectors. Without loss of generality we can assume that the columns of $T$ have unit Euclidian norm. The matrix $D$ can be interpreted as the Jordan form of the matrix $A$; thus, a diagonalizable matrix has (complex) Jordan blocks of the size $1\times1$ only.
It is clear that for a diagonalizable matrix $A$,
\begin{align*}
H_{A}(t)&=e^{At}=Te^{Dt}T^{-1},&t&>0,\\
H_{N,\tau,\alpha,A}(t)&=\sum_{n=0}^N TS_{n,\tau,\alpha,D}T^{-1}\,l_{n,\tau}^\alpha(t),&t&>0.
\end{align*}

Recall that we use the Frobenius norm in the space $\mathbb C^{M\times M}$.

\begin{theorem}\label{t:1}
Let $A\in\mathbb C^{M\times M}$ and $\lambda_k$, $k=1,2,\dots,M$, be eigenvalues of $A$. Then
\begin{equation}\label{e:est:below}
\sqrt{\max_{k}\zeta(N,\tau,\alpha,\lambda_k)}\le\lVert H_{A}-H_{N,\tau,\alpha,A}\rVert_{L_2[0,\infty)}
\end{equation}
and {\rm(}provided that $A$ is diaginalizable{\rm)}
\begin{equation}\label{e:est:above}
\begin{split}
\lVert H_{A}-H_{N,\tau,\alpha,A}\rVert_{L_2[0,\infty)}
&\le\varkappa(T)\sqrt{\sum_{k=0}^{M}\zeta(N,\tau,\alpha,\lambda_k)}\\
&\le\varkappa(T)\sqrt{M\max_{k}\zeta(N,\tau,\alpha,\lambda_k)},
\end{split}
\end{equation}
where $\varkappa(T)=\lVert T\rVert_{2\to2}\cdot\lVert T^{-1}\rVert_{2\to2}$ is the condition number~\cite[p.~63]{Higham08} of $T$.
\end{theorem}
\begin{proof}
Let $\lambda$ be an eigenvalue of $A$ and $v$ be the corresponding normalized eigenvector.
Since $H_{A}(t)$ and $S_{n,\tau,\alpha,\lambda}$ are respectively  the functions $\lambda\mapsto h_\lambda(t)$ and $\lambda\mapsto s_{n,\tau,\alpha,\lambda}$ of $A$ (Proposition~\ref{p:func cal}), $v$ is also the eigenvector of $H_{A}(t)$ and $S_{n,\tau,\alpha,\lambda}$, and it corresponds to the eigenvalues $h_{\lambda}(t)$ and $s_{n,\tau,\alpha,\lambda}$:
\begin{align*}
H_{A}(t)v&=h_{\lambda}(t)v, \\
H_{N,\tau,\alpha,A}(t)v&=\Bigl(\sum_{n=0}^{N}S_{n,\tau,\alpha,\lambda}\,l_{n,\tau}^\alpha(t)\Bigr)v
=\Bigl(\sum_{n=0}^{N}s_{n,\tau,\alpha,\lambda}\,l_{n,\tau}^\alpha(t)\Bigr)v.
\end{align*}

Therefore,
\begin{align*}
\bigl\lVert H_{A}-H_{N,\tau,\alpha,A}\bigr\rVert_{L_2} &\ge\bigl\lVert(H_{A}-H_{N,\tau,\alpha,A})v\bigr\rVert_{L_2}\\
&=\sqrt{\int_{0}^{\infty}\bigl\lVert \bigl(H_{A}(t)-H_{N,\tau,\alpha,A}(t)\bigr)v\bigr\rVert^2\,dt}\\
&=\sqrt{\int_{0}^{\infty}\Bigl\lVert \Bigl(h_{\lambda}(t)-\sum_{n=0}^{N}s_{n,\tau,\alpha,\lambda}\,l_{n,\tau}^\alpha(t)\Bigr)v\Bigr\rVert^2\,dt}\\
&=\sqrt{\int_{0}^{\infty}\Bigl\lvert h_{\lambda}(t)-\sum_{n=0}^{N}s_{n,\tau,\alpha,\lambda}(t)\,l_{n,\tau}^\alpha\Bigr\rvert^2
\cdot\lVert v\rVert^2\,dt}\\
&=\sqrt{\int_{0}^{\infty}\Bigl\lvert h_{\lambda}(t)-\sum_{n=0}^{N}s_{n,\tau,\alpha,\lambda}(t)\,l_{n,\tau}^\alpha    \Bigr\rvert^2}\,dt\\
&=\sqrt{\zeta(N,\tau,\alpha,\lambda)}.
\end{align*}
From this inequality, it follows estimate~\eqref{e:est:below}.

Estimate~\eqref{e:est:above} follows from Proposition~\ref{p:G_D-G_D,N} and the inequality
\begin{equation*}
\lVert TBT^{-1}\rVert_F\le\lVert T\rVert_{2\to2}\cdot\lVert B\rVert_F\cdot\lVert T^{-1}\rVert_{2\to2}=\varkappa(T)\cdot\lVert B\rVert_F.\qed
\end{equation*}
\renewcommand\qed{}
\end{proof}

\section{Derivatives with respect to $\tau$}\label{s:partial tau}
Derivatives with respect to $\tau$ of some of the involved functions have simple representations. This can help to find extreme points. In this section, we present relevant statements.

\begin{proposition}[{\rm see \cite{Belt-Brinker97,Brinker-Belt95,Clowes65}}]\label{p:Brinker:l}
We have
\begin{equation*}
\frac{\partial l_{n,\tau}^\alpha}{\partial\tau}(t)=d_{n+1}\,l_{n+1,\tau}^\alpha(t)
-d_{n}l_{n-1,\tau}^\alpha(t),\qquad n=0,1,\dots,
\end{equation*}
where $l_{-1,\tau}^\alpha(t)=0$ and
\begin{equation*}
d_0=0,\qquad d_{n}=\frac{\sqrt{n(n+\alpha)}}{2\tau}.
\end{equation*}
In particular, for $\alpha=0$,
\begin{equation*}
d_{n}=\frac{n}{2\tau}.
\end{equation*}
\end{proposition}
\begin{proof}
The proof follows from~\eqref{e:Laguerre functions:alpha} and the well-known~\cite[formula (5.1.14)]{Szego75:eng} formula
\begin{equation*}
\frac{\partial L_{n,t}^\alpha}{\partial t}=-L_{n-1,t}^{\alpha+1}.\qed
\end{equation*}
\renewcommand\qed{}
\end{proof}

\begin{corollary}\label{c:Brinker:s}
For Laguerre coefficients~\eqref{e:s_{n,tau,alpha,lambda}}, we have
\begin{equation*}
\frac{\partial s_{n,\tau,\alpha,\lambda}}{\partial \tau}=d_{n+1}\,s_{n+1,\tau,\alpha,\lambda}
-d_{n}\,s_{n-1,\tau,\alpha,\lambda},\qquad n=0,1,\dots.
\end{equation*}
\end{corollary}
\begin{proof}
It follows directly from~\eqref{e:s via int} and Proposition~\ref{p:Brinker:l}.
\end{proof}

\begin{corollary}\label{c:Brinker:xi}
For function~\eqref{e:func zeta}, we have
\begin{align*}
\frac{\partial \zeta(N,\tau,\alpha,\lambda)}{\partial\tau}
&=-d_{N+1}\bigl(s_{N+1,\tau,\alpha,\lambda}\,s_{N,\tau,\alpha,\bar\lambda}+s_{N+1,\tau,\alpha,\bar\lambda}\,s_{N,\tau,\alpha,\lambda}\bigr)\\
&=-2d_{N+1}\Real\bigl(s_{N+1,\tau,\alpha,\lambda}\,s_{N,\tau,\alpha,\bar\lambda}\bigr),
\end{align*}
where the bar over $\lambda$ means complex conjugate of $\lambda$.
\end{corollary}
\begin{proof}
We make use of representations~\eqref{e:zeta as  a sum from N+1} and~\eqref{e:bar s}:
\begin{equation*}
\zeta(N,\tau,\alpha,\lambda)=\sum_{n=N+1}^{\infty}|s_{n,\tau,\alpha,\lambda}|^2
=\sum_{n=N+1}^{\infty}s_{n,\tau,\alpha,\lambda}\,\overline{s_{n,\tau,\alpha,\lambda}}
=\sum_{n=N+1}^{\infty}s_{n,\tau,\alpha,\lambda}\,s_{n,\tau,\alpha,\bar\lambda}.
\end{equation*}
Differentiating the last formula, we obtain
\begin{equation*}
\frac{\partial \zeta(N,\tau,\alpha,\lambda}{\partial\tau}
=\sum_{n=N+1}^{\infty}
\frac{\partial s_{n,\tau,\alpha,\lambda}}{\partial\tau}\,s_{n,\tau,\alpha,\bar\lambda}+
s_{n,\tau,\alpha,\lambda}\,\frac{\partial s_{n,\tau,\alpha,\bar\lambda}}{\partial\tau}.
\end{equation*}
Then from Corollary~\ref{c:Brinker:s} it follows
\begin{align*}
\frac{\partial \zeta(N,\tau,\alpha,\lambda}{\partial\tau}
&=\sum_{n=N+1}^{\infty}
\bigl(d_{n+1}\,s_{n+1,\tau,\alpha,\lambda}
-d_{n}\,s_{n-1,\tau,\alpha,\lambda}\bigr)\,s_{n,\tau,\alpha,\bar\lambda}\\
&+
s_{n,\tau,\alpha,\lambda}\,\bigl(d_{n+1}\,s_{n+1,\tau,\alpha,\bar\lambda}
-d_{n}\,s_{n-1,\tau,\alpha,\bar\lambda}\bigr)\\
&=\sum_{n=N+1}^{\infty}
\bigl(d_{n+1}\,s_{n+1,\tau,\alpha,\lambda}\,s_{n,\tau,\alpha,\bar\lambda}
-d_{n}\,s_{n,\tau,\alpha,\bar\lambda}\,s_{n-1,\tau,\alpha,\lambda}\\
&+
d_{n+1}\,s_{n+1,\tau,\alpha,\bar\lambda}\,s_{n,\tau,\alpha,\lambda}
-d_{n}\,s_{n,\tau,\alpha,\lambda}\,s_{n-1,\tau,\alpha,\bar\lambda}\bigr).
\end{align*}
After canceling we obtain the desired representation.
\end{proof}

\section{The case $\alpha=0$}\label{s:alpha=0}
Our numerical experiments (see Section~\ref{s:num ex}) show that often the optimal value of $\alpha$ is close to 0. For this reason, we treated the case of $\alpha=0$ as a special one in the previous exposition. In this section, we collect some additional formulas related to $\alpha=0$. These formulas and Corollary~\ref{c:Brinker:xi} allow one to organize calculations for the case $\alpha=0$ substantially simpler and faster than for the general case. Thus, by taking $\alpha$ to be equal 0 (though the optimal $\alpha$ is only close to 0), we can take a larger number $N$ of terms in the truncated Laguerre series~\eqref{e:H_{N,tau,alpha,A}} and thereby compensate for the small loss of accuracy caused by a notoptimal value of $\alpha$.

\begin{proposition}\label{p:reccur r_n}
Let $\Real\lambda<0$. Then the Laguerre coefficients $s_{n,\tau,0,\lambda}$ can be calculated recursively{\rm :}
\begin{align*}
s_{0,\tau,0,\lambda}&=-\frac{2\sqrt{\tau }}{2\lambda-\tau},\\
s_{n+1,\tau,0,\lambda}&=\frac{2\lambda+\tau}{2\lambda-\tau}\,s_{n,\tau,0,\lambda}.
\end{align*}
\end{proposition}
\begin{proof}
It follows from Proposition~\ref{p:s_n}.
\end{proof}

\begin{corollary}\label{c:reccur r_n:func cal}
Let the spectrum of $A$ lie in the left half-plane. Then
the Laguerre coefficients $S_{n,\tau,0,A}$ can be calculated recursively{\rm :}
\begin{align*}
S_{0,\tau,0,A}&=-2\sqrt{\tau }(2A-\tau\mathbf1)^{-1},\\
S_{n+1,\tau,0,A}&=(2A+\tau\mathbf1)(2A-\tau\mathbf1)^{-1}\,S_{n,\tau,0,A}.
\end{align*}
\end{corollary}
\begin{proof}
The proof follows from Propositions~\ref{p:reccur r_n} and~\ref{p:func cal}.
\end{proof}

It is convenient to use Corollary~\ref{c:reccur r_n:func cal} for calculating $S_{n,\tau,0,A}$.

\begin{corollary}\label{c:zeta}
Let $\Real\lambda<0$. Then
\begin{align*}
\zeta(N,\tau,0,\lambda)
&=\frac{4\tau}{|2\lambda-\tau|^2}\cdot
\frac{\Bigl|\dfrac{2\lambda+\tau}{2\lambda-\tau}\Bigr|^{2N+2}} {1-\Bigl|\dfrac{2\lambda+\tau}{2\lambda-\tau}\Bigr|^2}\\
&=\frac{4\tau}{|2\lambda-\tau|^2-|2\lambda+\tau|^2}
\Bigl|\frac{2\lambda+\tau}{2\lambda-\tau}\Bigr|^{2N+2}.
\end{align*}
\end{corollary}
\begin{proof}
The proof immediately follows from Proposition~\ref{p:reccur r_n}.
\end{proof}

\section{The choice of $\tau$ and $\alpha$}\label{s:opt choice}
In this section, we propose an algorithm for choosing $\tau$ and $\alpha$.

Let a stable matrix $A$ be given. By means of the Jordan decomposition, we calculate eigenvalues and eigenvectors of $A$. Typically, at least due to rounding errors, the spectrum of $A$ is simple, moreover, all eigenvalues $\lambda_k$ are distinct. If the spectrum of $A$ is not simple, the proposed algorithm for choosing $\tau$ and $\alpha$ also works, but less can be said about the approximation accuracy.

We take a number $N\in\mathbb N$. For example, we take $N=10$. We do not use $N>50$ because of large rounding errors.

We consider the function
\begin{equation*}
\varphi(N,\tau,\alpha)=\sum_{k=0}^{M}\zeta(N,\tau,\alpha,\lambda_k),
\end{equation*}
where $\lambda_k$ are the eigenvalues of $A$ and $\zeta$ is defined by~\eqref{e:func zeta}.

First, we consider the case $\alpha=0$. By Corollary~\ref{c:Brinker:xi}, we have
\begin{equation*}
\frac{\partial \zeta(N,\tau,0,\lambda}{\partial\tau}
=-2d_{N+1}\Real\bigl(s_{N+1,\tau,0,\lambda}\,s_{N,\tau,0,\bar\lambda}\bigr).
\end{equation*}
From Proposition~\ref{p:s_n} we know that
\begin{equation*}
s_{n,\tau,0,\lambda}=-\frac{2\sqrt{\tau} (2\lambda+\tau)^n}{(2\lambda-\tau)^{n+1}},\qquad n=0,1,\dots.
\end{equation*}
Therefore,
\begin{equation*}
\frac{\partial\zeta(N,\tau,0,\lambda}{\partial\tau}
=-2d_{N+1}\Real\Bigl(
\frac{2\sqrt{\tau}(2\lambda+\tau)^{N+1}}{(2\lambda-\tau)^{N+2}}\,
\frac{2\sqrt{\tau}(2\bar\lambda+\tau)^{N}}{(2\bar\lambda-\tau)^{N+1}}\Bigr).
\end{equation*}
Finally, we arrive at
\begin{equation*}
\frac{\partial\varphi(N,\tau,0)}{\partial\tau}=-2d_{N+1}\sum_{k=0}^{M}
\Real\Bigl(
\frac{2\sqrt{\tau}(2\lambda_k+\tau)^{N+1}}{(2\lambda_k-\tau)^{N+2}}\,
\frac{2\sqrt{\tau}(2\bar\lambda_k+\tau)^{N}}{(2\bar\lambda_k-\tau)^{N+1}}\Bigr).
\end{equation*}
Numerical experiments show that the function $\tau\mapsto\frac{\partial\varphi(N,\tau,0)}{\partial\tau}$ increases. Hence the function $\tau\mapsto\varphi(N,\tau,0)$ has a unique minimum. We find it by solving the equation
\begin{equation*}
\frac{\partial\varphi(N,\tau,0)}{\partial\tau}=0
\end{equation*}
for $\tau$ (in `Mathematica'~\cite{Wolfram} it is done by the command \verb"FindRoot"; this command works iteratively; we take for the initial value $\tau=1$). It can be done quite quickly for values of $N$ up to 50. Thus, we find the optimal $\tau$ for the case $\alpha=0$. Let us denote it by $\tau_0$. After that we calculate $\varphi(N,\tau_0,0)$ using Corollary~\ref{c:zeta}  and the definition of $\varphi$.

Then we calculate symbolically $\varphi(N,\tau,\alpha)$ using formulas~\eqref{e:s_{n,tau,alpha,lambda}} and~\eqref{e:zeta as a sum from N+1}. Of course, the resulting formula is rather cumbersome.
We calculate (in `Mathe\-ma\-tica'~\cite{Wolfram} it is done by the command \verb"FindMinimum")
\begin{equation*}
\varphi_{\min}(N)=\min_{\tau>0,\,\alpha>-1}\varphi(N,\tau,\alpha).
\end{equation*}
We take only $N\le12$, because the calculations are notedly slow for greater $N$.
We take the found point of minimum $(\tau_1,\alpha_1)$ as the optimal values of $\tau$ and $\alpha$.
We use $\varphi_{\min}$ for the estimates of $\lVert H_{A}-H_{N,\tau,\alpha,A}\rVert_{L_2[0,\infty)}$ according to Theorem~\ref{t:1}; for the same aim, we also calculate
\begin{equation*}
\psi(N,\tau_1,\alpha_1)=\max_{\lambda_k\in\sigma(A)}\zeta(N,\tau_1,\alpha_1,\lambda_k)
\end{equation*}
for the found $\tau_1$ and $\alpha_1$.

Our numerical experiments (see Section~\ref{s:num ex}) show that the pair $\tau_0$ and $\alpha_0=0$ is often almost optimal. So, the consideration of $\alpha\neq0$ is not always necessarily.

\section{Numerical experiments}\label{s:num ex}
In this section, we describe two numerical examples.

\begin{example}\label{ex:1}
We consider the discrete model of a transmission line shown in fig~\ref{f:rcln}. We assume that the line consists of $n=150$ sections. Thus we have 150 unknown currents $I_C$ and 150 unknown voltages $U_L$. The parameters are as follows: $C=C_0/n$, $L=L_0/n$, $R=R_0/n$, $G=G_0/n$, where $C_0=10$, $L_0=50$, $R_0=170$, $G_0=160$. The modified nodal~\cite{Vlach-Singhal94} description of the circuit has the form $\dot x(t)=Ax(t)+f(t)$ with a matrix $A$ of the size $300\times300$. The spectrum of $-A$ is shown in fig.~\ref{f:line20}.

First, we take $N=10$ in the truncated Laguerre series~\eqref{e:H_{N,tau,alpha,A}}.
We calculate the minimum of $\varphi$ over $\tau$ for $\alpha=0$; as the initial value of $\tau$ we take $\tau=1$. We obtain the following results (see the left fig.~\ref{f:line20}). The optimal $\tau$ is $\tau_0=19.2$ (it is shown on the left fig.~\ref{f:line20} as a small square); $\sqrt{\varphi(10,\tau_0,0)}=0.00104$; $\varkappa(T)=28.4$; $\sqrt{M}=17.3$. According to Theorem~\ref{t:1} we have
\begin{align*}
\lVert H_{A}-H_{10,\tau_0,0,A}\rVert_{L_2[0,\infty)}&\ge\sqrt{\psi(10,\tau_0,0)}=0.000192,\\
\lVert H_{A}-H_{10,\tau_0,0,A}\rVert_{L_2[0,\infty)}
&\le\varkappa(T)\,\sqrt{\varphi(10,\tau_0,0)}=0.0294,\\
\lVert H_{A}-H_{10,\tau_0,0,A}\rVert_{L_2[0,\infty)}\
&\le\varkappa(T)\sqrt{M\,\psi(10,\tau_0,0)}=0.0945.
\end{align*}

Second, we repeat the same experiment with $N=30$. The result is as follows: $\tau_0$ is $19.3$, but $\sqrt{\varphi(30,\tau_0,0)}=4.21\cdot10^{-8}$, and
\begin{align*}
\lVert H_{A}-H_{30,\tau_0,0,A}\rVert_{L_2[0,\infty)}&\ge\sqrt{\psi(30,\tau_0,0)}=8.33\cdot10^{-9},\\
\lVert H_{A}-H_{30,\tau_0,0,A}\rVert_{L_2[0,\infty)}
&\le\varkappa(T)\,\sqrt{\varphi(30,\tau_0,0)}=1.83\cdot10^{-7},\\
\lVert H_{A}-H_{30,\tau_0,0,A}\rVert_{L_2[0,\infty)}\
&\le\varkappa(T)\sqrt{M\,\psi(30,\tau_0,0)}=4.09\cdot10^{-6}.
\end{align*}

Third, we return to $N=10$, take as initial values the found $\tau_0=19.2$ and $\alpha=0$, and find the minimum of $\varphi(N,\tau,\alpha)$ over $\tau$ and $\alpha$.
We obtain the following results (right fig.~\ref{f:line20}). The optimal $\tau$ is $\tau_1=19.20$; the optimal $\alpha$ is $\alpha_1=0.0000239$; $\sqrt{\varphi(10,\tau_1,\alpha_1)}=0.001036$; $\varkappa(T)=28.358$; $\sqrt{M}=17.3$. According to Theorem~\ref{t:1} we have
\begin{align*}
\lVert H_{A}-H_{10,\tau_1,\alpha_1,A}\rVert_{L_2[0,\infty)}&\ge\sqrt{\psi(10,\tau_1,\alpha_1)}=0.000193,\\
\lVert H_{A}-H_{10,\tau_1,\alpha_1,A}\rVert_{L_2[0,\infty)}
&\le\varkappa(T)\,\sqrt{\varphi(10,\tau_1,\alpha_1)}=0.0294,\\
\lVert H_{A}-H_{10,\tau_1,\alpha_1,A}\rVert_{L_2[0,\infty)}\
&\le\varkappa(T)\sqrt{M\,\psi(10,\tau_1,\alpha_1)}=0.0947.
\end{align*}
\end{example}

\begin{figure}[htb]
\begin{center}
\includegraphics[width=.9\textwidth]{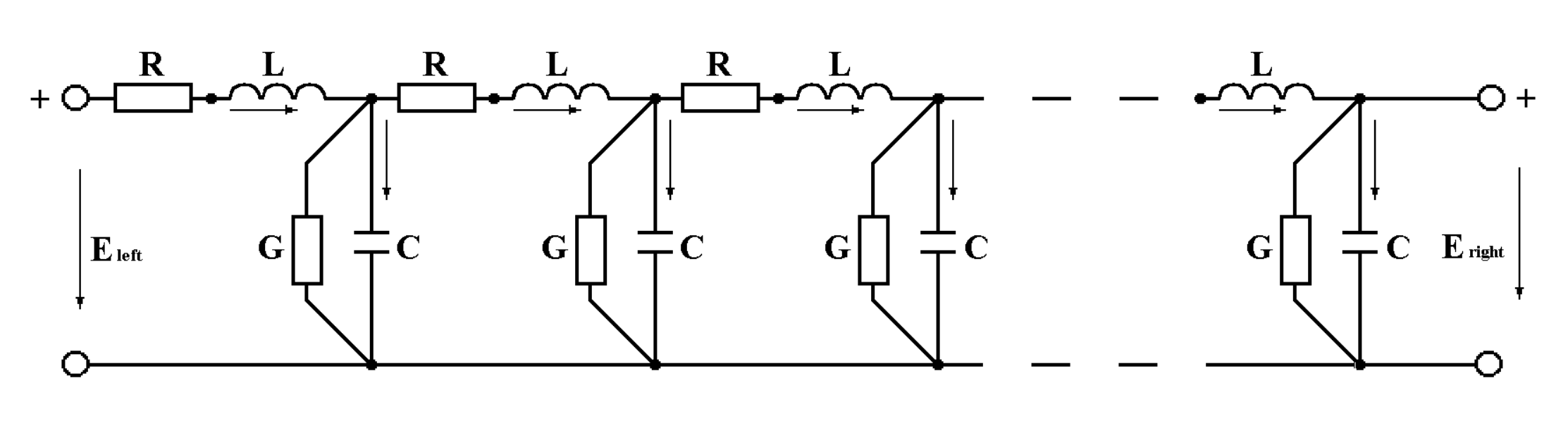}
\caption{A discrete model of a transmission line}\label{f:rcln}
%
\includegraphics[width=\textwidth]{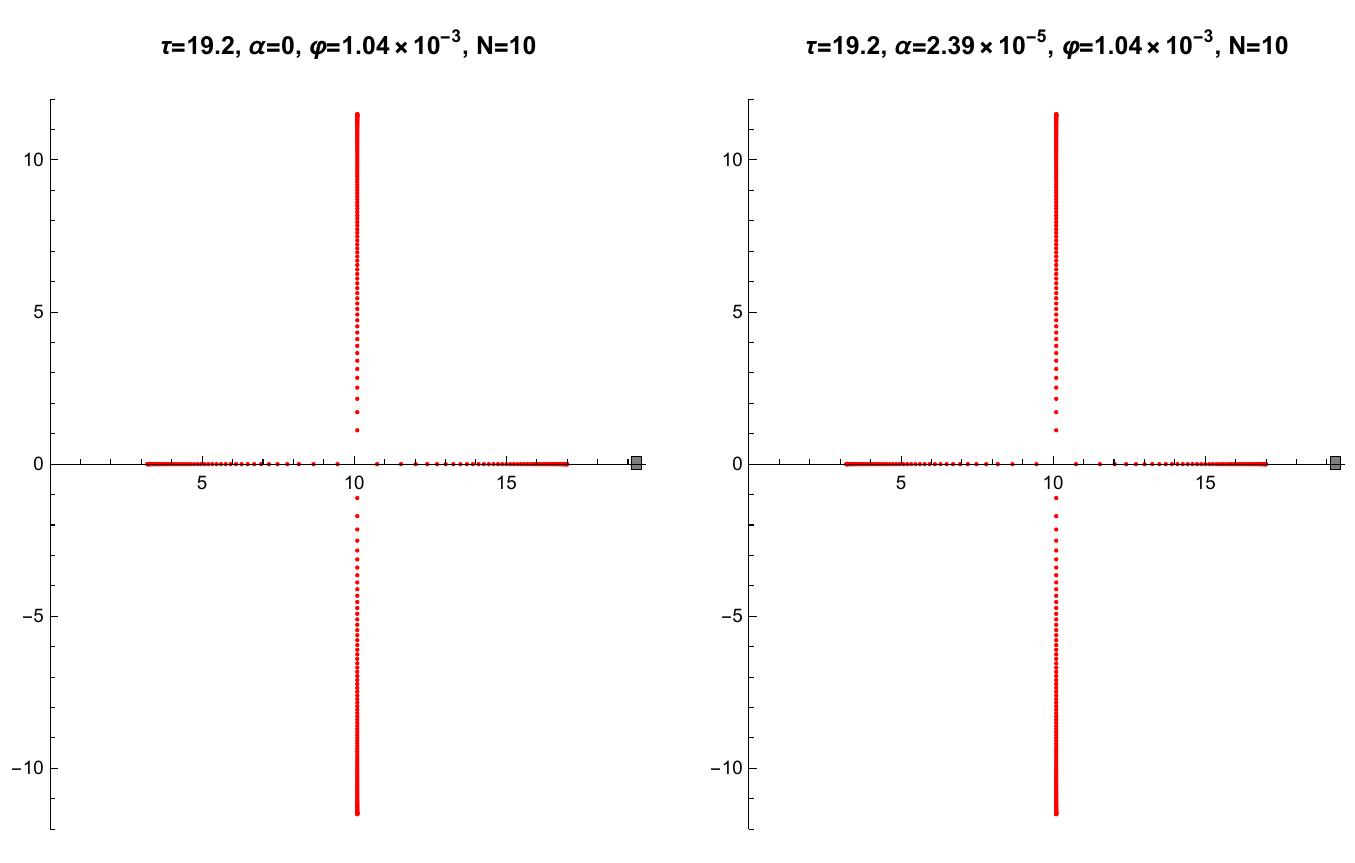}
\caption{The points show the spectrum of the matrix $-A$ from example~\ref{ex:1}, the small squares are the found $\tau$. In the right figure the minimum is taken over $\tau$ and $\alpha$; in the left figure the minimum is taken only over $\tau$ with $\alpha=0$}\label{f:line20}
\end{center}
\end{figure}

\begin{example}\label{ex:2}
In Example~\ref{ex:1} the minimum of $\varphi(N,\tau,\alpha)$ is attained almost at $\alpha=0$.
We present here another example of the same kind. Since our estimate depends only on the spectrum of $A$, we do not present a matrix $A$ itself and work only with its possible spectrum.

We take 200 random complex numbers; their real parts have the Maxwell distribution with $\sigma=4$ (the probability density for value $x$ in the Maxwell distribution is proportional to $x^2 e^{-x^2/(2\sigma^2)}$ for $x>0$, and is zero for $x<0$) and imaginary parts have the normal distribution with the mean value $\mu=0$ and the variance $\sigma^2=1$. We interpret these points as a possible spectrum of $-A$; we present them in fig.~\ref{f:prob20}.
The results of calculation are as follows.

First we take $N=10$ and $\alpha=0$. Starting from the initial point $\tau=1$ we find that the minimum of $\varphi(10,\tau,0)$ over $\tau$ is attained at $\tau_0=5.83$ and
\begin{align*}
\sqrt{\varphi(10,\tau_0,0)}&=0.00878,\\
\sqrt{\psi(10,\tau_0,0)}&=0.00742.
\end{align*}
The experiment with $N=30$ and $\alpha=0$ gives $\tau_0=5.67$ and
\begin{align*}
\sqrt{\varphi(30,\tau_0,0)}&=1.58\cdot10^{-6},\\
\sqrt{\psi(30,\tau_0,0)}&=1.32\cdot10^{-6}.
\end{align*}

Then we take $N=10$ and begin iterations from the found $\tau_0=5.83$ and $\alpha=0$. Now
the optimal $\tau$ is $\tau_1=5.82$ and the optimal $\alpha$ is $\alpha_1=-0.000915$. For the estimates from Theorem~\ref{t:1} we have
\begin{align*}
\sqrt{\varphi(10,\tau_1,\alpha_1)}&=0.00874,\\
\sqrt{\psi(10,\tau_1,\alpha_1)}&=0.00738.
\end{align*}

\begin{figure}[htb]
\begin{center}
\includegraphics[width=\textwidth]{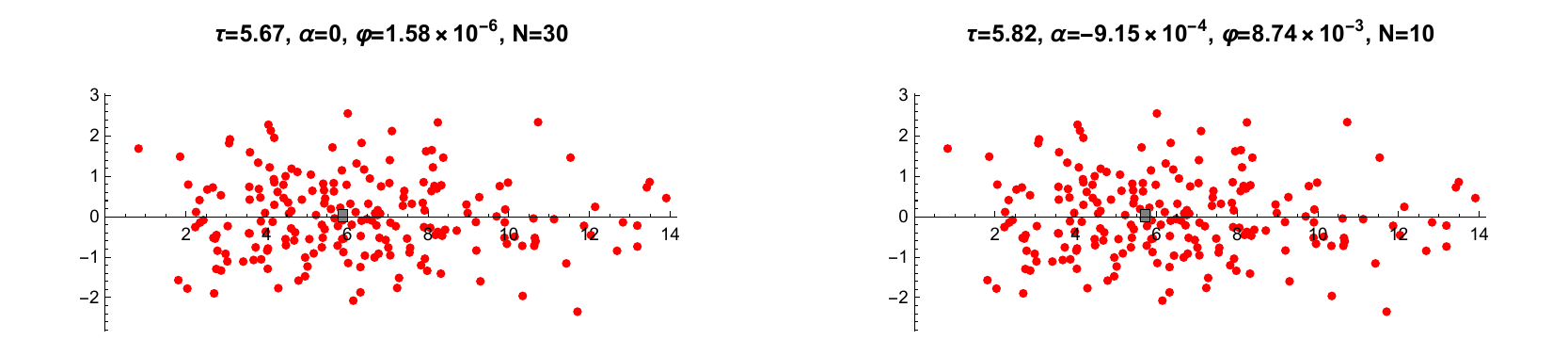}
\caption{The points from example~\ref{ex:2}, the small squares are the found $\tau$. In the right figure the minimum is taken over $\tau$ and $\alpha$; in the left figure the minimum is taken only over $\tau$ with $\alpha=0$. Note that in the right fig. $N=10$, but in the left fig. $N=30$}\label{f:prob20}.
\includegraphics[width=.45\textwidth]{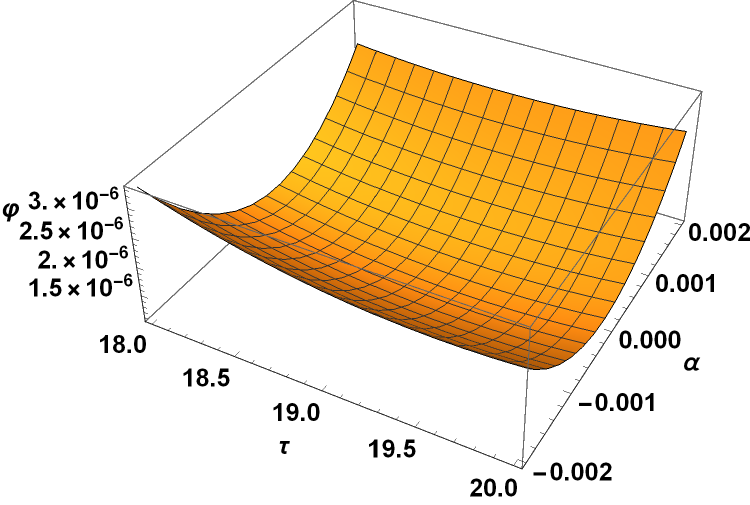}\hfill
\includegraphics[width=.45\textwidth]{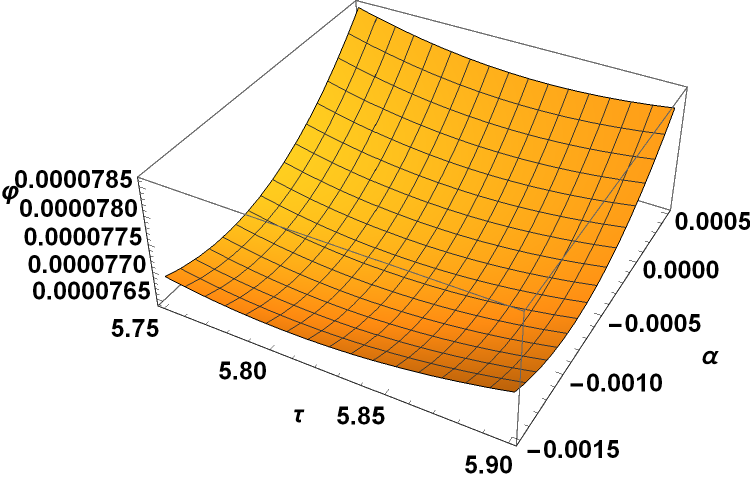}
\caption{The graphs of the function $\varphi$ from Examples~\ref{ex:1} (left) and~\ref{ex:2} (right) for $N=10$ in a neighbourhood of the minimum point}\label{f:plot_phi}
\end{center}
\end{figure}

Fig.~\ref{f:plot_phi} shows that the function $\varphi$ is rather smooth and convex. Thus, the problem of finding of its minimum can be solved by standard tools.
\end{example}

\end{document}